\documentclass[a4paper, 11pt, english, oneside, reqno]{article}

\usepackage[top=1.25in, bottom=1.25in, left=1.25in, right=1.25in]{geometry}
\usepackage[utf8]{inputenc}
\usepackage{amsfonts, amssymb, amsmath, amsthm, wasysym}
\usepackage{mathrsfs}
\usepackage{graphicx}
\usepackage{xcolor}
\definecolor{blue-green}{rgb}{0.0, 0.85, 0.85} 
\definecolor{some-blue}{rgb}{0.12, 0.56, 1.0} 
\definecolor{chocolate}{rgb}{0.48, 0.25, 0.0}

\usepackage[toc,page]{appendix}
\usepackage{caption}

\usepackage{transparent}
\usepackage{eepic}
\usepackage{tocbibind} 

\usepackage{tikz} 

\usepackage{accents}
\usepackage{setspace}
\usepackage{footnote}
\usepackage[colorlinks = true]{hyperref}
\hypersetup{
	pdfauthor={Javier Minguillón},
	pdftitle={A Note on Almost Everywhere Convergence Along Tangential Curves to the Schrödinger Equation Initial Datum},
	pdfkeywords={Schrödinger equation, pointwise convergence, nontangential convergence},
	pdfcreator={LaTeX with hyperref},
	pdfproducer={pdfLaTeX}
	colorlinks=true,
	linkcolor={blue!80!black},
	citecolor={blue!50!black},
	urlcolor={blue!50!black}
}

\def\R{\mathbb R}
\def\Z{\mathbb Z}

\usepackage{accents}

\newcommand{\largetwoN}{2N}
\newcommand{\calpha}{\tau}
\newcommand{\GammaR}[1]{\Gamma^{\alpha}\left(#1\right)} 

\newcommand{\schtf}[1]{e^{it\Delta}f \left(#1\right)}
\newcommand{\dxt}{dx 
	dt}

\usepackage{xparse}

\let\latexchi\chi
\makeatletter
\renewcommand\chi{\@ifnextchar_\sub@chi\latexchi}
\newcommand{\sub@chi}[2]{
	\@ifnextchar^{\subsup@chi{#2}}{\latexchi^{}_{#2}}%
}
\newcommand{\subsup@chi}[3]{
	\latexchi_{#1}^{#3}%
}
\makeatother

\newcommand{\verluego}[1]{{\textsc{\color{red} 
}}}
\newcommand{\verluegointerno}[1]{{\color{chocolate}
}}

\newcommand{\verluegotres}[1]{{{\color{blue} 
}}}
\newcommand{\verluegocuatro}[1]{{\color{black} 
}}
\newcommand{\verluegoalpha}[1]{{\color{blue-green} 
}}
\newcommand{\verluegoalphados}[1]{{\color{some-blue} 
}}

\newtheorem{theorem}{Theorem}
[section]

\theoremstyle{definition}
\newtheorem{defin}[theorem]{Definition}

\newtheorem{remark}[theorem]{Remark}

\usepackage[
backend=biber,
giveninits=true, 
style=numeric, 
]{biblatex}

\DeclareFieldFormat[article]{author}{\textsc{#1}}
\DeclareFieldFormat[article]{title}{\textit{#1}}
\DeclareFieldFormat[article]{journaltitle}{{#1}}
\DeclareFieldFormat[article]{volume}{{#1}}
\DeclareFieldFormat[article]{number}{{ no. #1}}
\DeclareFieldFormat[article]{pages}{#1}
\DeclareFieldFormat[article]{year}{(#1)}

\DeclareFieldFormat[book]{title}{{#1}}
\DeclareFieldFormat[book]{publisher}{#1}

\DeclareFieldFormat[online]{title}{\textit{#1}}
\DeclareFieldFormat[online]{date}{(#1)} 
\DeclareFieldFormat[online]{url}{\url{#1}}

\DeclareFieldFormat[incollection]{title}{#1}

\AtEveryBibitem{%
	\clearfield{month}%
	\clearfield{day}%
	\clearfield{urldate}%
		\clearfield{doi} 
	\clearfield{eprintclass}
	\clearfield{url}
	\clearfield{pubstate}
	\clearfield{editor}
	\clearfield{isbn}
	\clearfield{issn}
	\clearfield{journaltitle}  
}

\renewbibmacro*{journal}{%
	\ifboolexpr{
		test {\iffieldundef{shortjournal}}
		and
		test {\iffieldundef{journaltitle}}
	}
	{}
	{\printtext[journaltitle]{%
			\printfield[titlecase]{shortjournal}%
			\setunit{\subtitlepunct}%
			\printfield[titlecase]{journaltitle}}}
}

\renewbibmacro*{url+urldate}{%
	\printfield{url}%
}

\renewbibmacro{in:}{} 

\date{}
\title{\sc A note on almost everywhere convergence along tangential curves to the Schrödinger equation initial datum}
\author{\sc Javier Minguillón}

\begin{document}
	\maketitle
		\begin{abstract}
		\begin{spacing}{1.4} 
			In this short note, we give an easy proof of the following result:  for $ n\geq 2, $ 
			$\underset{t\to0}{\lim} \,e^{it\Delta }f\left(x+\gamma(t)\right) = f(x) $
			almost everywhere whenever $ \gamma $ is an $ \alpha- $Hölder curve with $ \frac12\leq \alpha\leq 1 $ and $ f\in H^s(\R^n) $, with $ s > \frac{n}{2(n+1)} $. This is the optimal range of regularity up to the endpoint.
		\end{spacing}
	\end{abstract}
	
	\footnote{
		Javier Minguillon was supported by Grant PID2022-142202NB-I00 / AEI / 10.13039/501100011033.
	}

\section{Introduction}
Consider the linear Schrödinger equation on $ \R^n\times \R $, $ n\geq 1, $ given by
\begin{equation}
	\label{equation_schrodinger with boundary datum}
	\begin{cases}
		i\partial_t u(x,t) - \Delta_x u (x,t) = 0,
		\\
		u(x,0) = f(x).
	\end{cases}
\end{equation} 
Its solution can be formally expressed as
\begin{equation}
	\schtf{x} = \int_{\R^n} e^{2\pi i x\cdot \xi} e^{2\pi it|\xi|^2} \widehat f(\xi) d\xi. 
\end{equation}
It was first proposed by Carleson in 1980 \cite{Carleson1980AnalyticProblemsRelated} to find the values of $ s>0 $ for which
\begin{equation}\label{equation_classical ae convergence}
	\lim_{t\to 0} \schtf{x} = f(x), \quad \text{a.e.}\quad x\in \R^n,
\end{equation}
holds true for all functions $ f\in H^s(\R^n) $. Carleson \cite{Carleson1980AnalyticProblemsRelated} proved this convergence when $ n=1 $ and $ s\geq \frac14$. Later, in 2006, Dahlberg and Kenig \cite{DahlbergKenig2006NoteAlmostEverywhere} showed that 
\eqref{equation_classical ae convergence} 
was false whenever $ s<\frac14. $ 

Many researchers have worked in this problem throughout the years. Authors such as Carbery, Cowling, Vega, Sjölin, Moyua, Vargas, Tao, S.Lee and Bourgain to name a few. 
More recently, the problem has been solved in higher dimensions, except for the endpoint. In 2016, Bourgain \cite{Bourgain2016NoteSchrodingerMaximal} proved the necessity of $ s\geq \frac{n}{2(n+1)} $ in order to have 
\eqref{equation_classical ae convergence}. In 2017, Du, Guth and Li \cite{DuGuthLi2017SharpSchrodingerMaximal} proved the sufficiency of the condition $ s>\frac13 $ when $ n=2. $
Later, in 2019, Du and Zhang \cite{DuZhang2019SharpEstimatesSchrodinger} proved the sufficiency of $ s>\frac{n}{2(n+1)} $ for general $ n\geq 3. $ A more detailed history of the problem can be found in \cite{DuZhang2019SharpEstimatesSchrodinger} and the references therein.

Take a solution of \eqref{equation_schrodinger with boundary datum}. Consider a set of curves $ \rho(x,t) = x + \gamma(t) $ that are bi-Lipschitz in $ x\in \R^n $ and $ \alpha- $Hölder in $ t\in\R $. Cho, Lee and Vargas \cite{ChoLeeVargas2012ProblemsPointwiseConvergence}
proved in 2012 that $ u\left(\rho(x,t),t\right) $ converges to $ f(x) $ almost everywhere as $ t\to 0 $
in $ n=1 $ when $ s>\max\left\lbrace \frac12-\alpha, \frac14 \right\rbrace $. They also found this to be sharp up to the endpoint. Later, in 2021, Li and Wang \cite{LiWang2021ConvergencePropertiesGeneralized} proved that convergence in dimension $ n = 2 $, for index $ \frac12\leq \alpha\leq 1 $ and the range $ s > \frac38 $. 
In 2023, Cao and Miao \cite{CaoMiao2023SharpPointwiseConvergence} gave a proof for general dimension $ n $, index $ \frac12\leq \alpha\leq 1 $, and $ s>\frac{n}{2(n+1)} $. Their proof followed the argument presented in \cite{DuZhang2019SharpEstimatesSchrodinger} and relied on techniques such as 
dyadic pigeonholing, broad-narrow analysis
and induction on scales.

Our objective is to give an easy proof of the result in \cite{CaoMiao2023SharpPointwiseConvergence} without using the aforementioned techniques.

%
Fix $0<\alpha\leq 1$ and $\calpha\geq1.$ We consider the family of curves, 
\begin{equation}
	\Gamma_{\calpha}^\alpha 
	:= 
	\left\lbrace
	\gamma:[0,1] \to \mathbb R^n 
	: \text{for all } t,t'\in [0,1],\; |\gamma(t)-\gamma(t')| \leq \calpha|t-t'|^\alpha
	\right\rbrace.
\end{equation} 
The convergence result follows from the maximal bound below. 
Let $ B_r^n(x_0) $ denote the ball of radius $ r>0  $ centered at $ x_0\in\R^n $.

\begin{theorem}\label{theorem_Theorem 1.2 of Du-Zhang modified}

	Let $n \geq 1$. Fix $ \frac12\leq  \alpha\leq 1 $ and $\calpha\geq1 $. 
	For any $\varepsilon>0$, there exists a positive constant $C_{\varepsilon,\calpha}$ such that, for every $ \gamma\in\Gamma_{\calpha}^\alpha $,
	\begin{equation}\label{equation_maximal estimate non-reduced}
		\left\|
		\sup _{0<t<1}\left| 
		\schtf{x + \gamma(t)}
		\right|
		\right\|_{L^2\left(B^{n}_1(0)\right)} 
		\leq 
		C_{\varepsilon,\calpha}
		\|f\|_{
			H^{\frac{n}{2(n+1)} \verluegoalphados{+\left(n+2\right)(1-2\alpha)}
				+\epsilon}
			\left(\R^n\right)
		},
	\end{equation}
	holds for all 
	$ f $
	$ \in H^{\frac1{2(n+1)}+\epsilon }(\R^n). $
\end{theorem}

\begin{remark}
	A change of variables shows that it is enough to consider the case $ \tau = 1 $. From now on we assume $ \tau = 1 .$
\end{remark}
Then, we can reduce Theorem \ref{theorem_Theorem 1.2 of Du-Zhang modified} as in \cite{DuZhang2019SharpEstimatesSchrodinger}.
We begin with a definition. 

\begin{defin}
	Fixed $0<\alpha\leq 1$ and $ R>1 $, we define
	\begin{equation}
		\GammaR{R^{-1}}
		:= 
		\left\lbrace
		\gamma:[0,R^{-1}] \to \mathbb R^n 
		: \text{for all } t,t'\in [0,R^{-1}],\; |\gamma(t)-\gamma(t')| \leq |t-t'|^\alpha
		\right\rbrace.
	\end{equation} 
\end{defin}
By Littlewood-Paley decomposition, the time localization lemma (e.g. Lemma 3.1 in S.Lee \cite{Lee2006PointwiseConvergenceSolutions}) and parabolic rescaling, Theorem \ref{theorem_Theorem 1.2 of Du-Zhang modified} can be reduced to the following Theorem \ref{theorem_Theorem 1.3. from Du Zhang modified}.

\begin{theorem}\label{theorem_Theorem 1.3. from Du Zhang modified}
	\verluegointerno{\textbf{Theorem 1.3 from Du Zhang modified.}}
	Let $n \geq 1$ and $ \frac12\leq \alpha <1 $. 
	For any $\varepsilon>0$, there exists a constant $C_{\varepsilon}$ such that, for all $ \gamma\in \GammaR{R^{-1}} $,
	\begin{equation}
		\left\|
		\sup _{0<t \leq R}\left| 
		e^{i t \Delta} f 
		\left(
		x + R\gamma\left(\frac{t}{R^2}\right)
		\right)
		\right|
		\right\|_{L^2\left(B^n_R(0)\right)} 
		\leq 
		C_{\varepsilon} R^{\frac{n}{2(n+1)}+\varepsilon}
		\verluegoalphados{R^{(n+2)(1-2\alpha)}}
		\|f\|_2.
	\end{equation}
	holds for all $R \geq 1$ and all $f$ with $\operatorname{supp} \widehat{f} \subset A(1)=\left\{\xi \in \mathbb{R}^n:|\xi| \sim 1\right\}$. 
\end{theorem}

\section{Intermediate results} 

	We consider the following result from \cite{DuZhang2019SharpEstimatesSchrodinger}.
	
	\begin{theorem}[Corollary 1.7 in \cite{DuZhang2019SharpEstimatesSchrodinger}]\label{theorem_Cor1.7 of Du Zhang}
		Let $n \geq 1$. For any $\varepsilon>0$, there exists a constant $C_{\varepsilon}$ such that the following holds for all $R \geq 1$ and all $f$ with $\operatorname{supp} \widehat{f} \subset B_1^n(0)$. Suppose that $X=\cup_k B_k$ is a union of lattice unit cubes in $B_R^{n+1}(0)$. Let $1 \leq \beta \leq n+1$ and
		
		\begin{equation}\label{equation_beta density of X}
			\phi:= \phi_{X,\beta}:=
			\max _{\substack{ B_r^{n+1}(x') \subset B_R^{n+1}(0) \\ x^{\prime} \in \mathbb{R}^{n+1}, r \geq 1}} 
			\frac{\#\left\{B_k: B_k \subset B^{n+1}_r\left(x^{\prime} \right)\right\}}{r^\beta} .
		\end{equation}
		Then
		\begin{equation}
			\left\|e^{i t \Delta} f\right\|_{L^2(X,\dxt)} 
			\leq C_{\varepsilon} \phi^{\frac{1}{n+1}} R^{\frac{\beta}{2(n+1)}+\varepsilon}\|f\|_2.
		\end{equation}
		
	\end{theorem}
	
	We generalize the above result to include $ \alpha- $Hölder curves. 

	\begin{theorem}\label{theorem_Cor1.7 of Du Zhang modified}
		\verluegointerno{\bf Corollary 1.7 from Du Zhang modified.}
		Let $n \geq 1$ and $ \frac12\leq \alpha \leq 1 $. 
		For any $\varepsilon>0$, there exists a constant $C_{\varepsilon}$ such that the following holds for any $R \geq 1$, every $ \gamma\in\GammaR{R^{-1}} $ and all $f$ with $\operatorname{supp} \widehat{f} \subset B_1^n(0)$. Suppose that $X=\cup_k B_k$ is a union of lattice unit cubes in $B_R^{n+1}(0)$. Let $1 \leq \beta \leq n+1$ and $\phi$ be given by 
		\eqref{equation_beta density of X}.
		Then
		
		\begin{equation}
			\left\|e^{i t \Delta} f\left(x + R\gamma\left(\frac{t}{R^2}\right)\right)\right\|_{L^2(X, \dxt)} 
			\leq 
			\verluegoalphados{R^{1-2\alpha}}
			C_{\varepsilon} \phi^{\frac{1}{n+1}} R^{\frac{\beta}{2(n+1)}+\varepsilon}\|f\|_2. 
		\end{equation}
	\end{theorem}

	\begin{proof}[Proof of Theorem \ref{theorem_Cor1.7 of Du Zhang modified}]
		Denote 
		\begin{equation}\label{equation_notation theta to be rid of gamma}
			\theta(t) := \theta_{R
			} 
			(t) := R\gamma \left(\frac t{R^2}\right).
		\end{equation} 
		We begin with
		\begin{align}
			\left\| 
			\schtf{x + R\gamma \left(\frac{t}{R^2}\right)}
			\right\|_{L^2(X, \dxt)}^2
			&
			= 
			\sum_{k}
			\int_{B_k} 
			\left|\schtf{x+\theta(t)}\right|^2
			\dxt.
			\intertext{
				Denote $ (x_k,t_k)$ to be the center of $ B_k. $ Then,
			}
			&\leq
			\sum_{k}
			\int_{t_k-1\verluegoalpha{\cdot R^{(2\alpha-1)}}}^{t_k+1\verluegoalpha{\cdot R^{(2\alpha-1)}}}
			\int_{B^n_1(x_k)} 
			\left|\schtf{x+\theta(t)}\right|^2 
			dxdt
			\\
			&= 
			\sum_{k} 
			\int_{t_k-1}^{t_k+1}
			\int_{B^n_1\left(x_k+\theta(t)\right)}
			\left|\schtf{y}\right|^2 
			dydt.
			\intertext{
				Recall that $ \gamma\in \GammaR{R^{-1}} $, and $ \alpha\geq \frac12 $. Thus, if $ t\in (t_k-1,t_k+1), $ then $ \left|\theta(t)-\theta(t_k)\right| \leq 1 $. Hence,
			}
			&\leq 
			\sum_{k}
			\int_{t_k-1}^{t_k+1}
			\int_{B^n_3\left(x_k+\theta(t_k)\right)}
			\left|\schtf {y} \right|^2
			dydt
			\\
			&
			\leq 
			\int_{\R^n}\sum_k \chi_{B^{n+1}_4\left(x_k + \theta(t_k),t_k\right)}(y) \left|\schtf{y}\right|^2dydt
			\\ 
			&
			\leq 
			C
			\int_{\bigcup B^{n+1}_4(x_k + \theta(t_k),t_k)} \left|\schtf{y}\right|^2dydt,
		\end{align}
		for some $ C>0. $ Note that, if $ B_4(x_k + \theta(t_k), t_k) \cap B_4(x_i + \theta(t_i), t_i) \neq \emptyset, $ then $ |t_k-t_i|\leq 8. $ Since $ \alpha\geq \frac12,  $ we have $ |\theta(t_k)-\theta(t_i)| \leq 8 $. Hence, $ |x_k-x_i| \leq 16. $ Therefore, the balls $ \{B_4(x_k+\theta(t_k),t_k)\}_k $ have finite overlap $ C = C_n  $.
		
		Define 
		$ Y = \bigcup_l Q_l $ to be the minimal union of lattice unit cubes satisfying that $ \bigcup_k B_4^{n+1}(x_k + \theta(t_k), t_k)\subset Y. $ We have proven that
		\begin{align}
			\left\|\schtf{x+\theta(t)}\right\|_{L^2(X, \dxt)}
			&\leq
			C_n 
			\left\|\schtf{x}\right\|_{L^2(Y,\dxt)}.
			\intertext{
				Hence by Theorem \ref{theorem_Cor1.7 of Du Zhang},} 
			&\leq 
			C_{\epsilon,n}
			\phi_{Y,\beta}^{\frac1{n+1}}R^{\frac{\beta}{2(n+1)}+\epsilon} \|f\|_2.
		\end{align}
		We \textbf{claim} that  
		\begin{equation}\label{claim-densities are comparable}
			\phi_{Y,\beta} 
			\leq
			c_n
			\phi_{X,\beta}, 
		\end{equation} for some $ C>0. $ This would conclude the proof. 
		
		To prove \eqref{claim-densities are comparable}, note that, if $ Q_l \subset B_r^{n+1}(y_0,s_0) $, $ r\geq1, $ and $ Q_l\cap B_4^{n+1}(x_k + \theta(t_k),t_k) \neq \emptyset $, then $ B_k = B_1^{n+1}(x_k,t_k) \subset B_{r+5}^{n+1} (y_0-\theta(s_0), s_0). $
		Therefore,
		\begin{equation}
			\dfrac{\#\{Q_l : Q_l \subset B_r^{n+1}(y_0,s_0)\}}{r^{\beta}} 
			\leq 
			c_n
			\dfrac{
				\#\left\lbrace
				B_k : B_k \subset B_{r+5}^{n+1} (y_0 -\theta(s_0), s_0)
				\right\rbrace
			}
			{(r+5)^\beta}
			\cdot 
			\dfrac{(r+5)^\beta}{r^\beta}
			\leq c_n \phi_{X,\beta}.
		\end{equation}
	\end{proof}
\section{Proof of Theorem \ref{theorem_Theorem 1.3. from Du Zhang modified}} 
	Before the proof, let us introduce a stability property of the Schrödinger operator. More general versions of the following appeared in an article of Tao \cite{Tao1999BochnerRieszConjectureImpliesa} from 1999 and an article of Christ \cite{Christ1988RegularityInversesSingular} from 1988. 
	
	Suppose that 
	$ \widehat f $ is supported inside a ball of radius 1. 
	If $ |x'-y'| \leq 4 $ and $ |t'-s'|\leq 4, $ then,
	\begin{equation}
		\left|e^{it'\Delta}f(x')\right|
		\leq
		\sum_{\mathfrak{l}\in\Z^n} 
		\frac{1}{(1+|\mathfrak{l}|)^{n+1}} \left|e^{is'\Delta}f_\mathfrak{l}(y')\right|,
	\end{equation}
	where $ \widehat{f_\mathfrak{l}}(\xi) = e^{2\pi i\mathfrak{l}\xi}\widehat f(\xi)$.
	
	Now, fix $ \alpha \geq 1/2 $ and $ \gamma\in \GammaR{R^{-1}}$. Define $ \theta $ as in \eqref{equation_notation theta to be rid of gamma}. Whenever $ |x-y|\leq 2 $ and $ |t-s|\leq 2 $, we have that 
	$ \left| x+\theta(t) - ( y + \theta(s)) \right|\leq 4. $ Thus,
	\begin{equation}
		\left|e^{it\Delta}f(x+\theta(t))\right|
		\leq
		\sum_{\mathfrak{l}\in\Z^n} 
		\frac{1}{(1+|\mathfrak{l}|)^{n+1}} \left|e^{is\Delta}f_\mathfrak{l}(y+ \theta(s))\right|.
	\end{equation}
	Therefore, if $ |x-x_0| \leq 1 $ and $ |t-t_0|\leq 1 $, then,
	\begin{equation}\label{equation_stability lemma aka locally constant property substitute}
		\left|e^{it\Delta}f(x+\theta(t))\right| 
		\leq
		\sum_{\mathfrak{l}\in\Z^n} 
		\frac{1}{(1+|\mathfrak{l}|)^{n+1}} \int_{t_0}^{t_0+1}\int_{B_1(x_0)}\left|e^{is\Delta}f_\mathfrak{l}(y+ \theta(t))\right|dyds.
	\end{equation}
	
	\begin{proof}
		[Proof of Theorem \ref{theorem_Theorem 1.3. from Du Zhang modified}.]
		For the sake of briefness, given $ (x,t)\in\R^{n+1} $ let us denote 
		\begin{equation}
			E'f(x,t) 
			:= E'_{\gamma,R} f(x,t) 
			:= e^{i t \Delta} f 
			\left(
			x + R\gamma\left(\frac{t}{R^2}\right)
			\right).
		\end{equation}
		Now, we can write 
		\begin{align}
			&\left\|
			\sup _{0<t \leq R}
			\left| 
			e^{i t \Delta} f 
			\left(
			x + R\gamma\left(\frac{t}{R^2}\right)
			\right)
			\right|
			\right\|_{L^2\left(B^n_R(0)\right)} ^2
			=
			\left\|
			\sup _{0<t \leq R} |E'f(x,t)|
			\right\|_{L^2\left(B^n_R(0)\right)} ^2
			\\ 
			&
			=
			\int_{B^n_R(0)} \left(
			\sup_{\substack{0<t\leq R} }|E'f(x,t)|^2 
			\right)dx
			\\
			&
			=
			\sum_{\substack{x_0\in\Z^n \\ |x_0|<R} }
			\int_{B^n_1(x_0)} 
			\left(
			\sup _{0<t \leq R} |E'f(x,t)|^2
			\right)
			dx
			\\
			&
			\leq C_n 
			\sum_{\substack{x_0\in\Z^n \\ |x_0|<R} }
			\left[
			\sup_{|x-x_0| \leq 1}
			\left(
			\sup_{\substack{0 < t \leq R}} 
			|E'f(x,t)|^2
			\right)
			\right].
				%
		\end{align}
		For each $ x_0 $, there exists $\widetilde{t_0} = \widetilde{t_0}(x_0)\in \Z\cap[0,R] $ such that the supremum on each term of the above sum is almost attained inside $ B^{n+1}_1(x_0,\widetilde{t_0}) $. Therefore, 
		\begin{align}
			\left\|
			\sup_{0< t\leq R}|E'f(x,t)|\right\|_{L^2(B^n_R(0))}^2
			&
			\lesssim 
			\sum_{\substack{x_0\in\Z^n \\ |x_0|<R} } \sup_{(x,t)\in B_1^{n+1}(x_0,\widetilde{t_0})} \left|E'f(x,t)\right|^2,
			\intertext{which is, by \eqref{equation_stability lemma aka locally constant property substitute},}
			&
			\lesssim 
			\sum_{\substack{x_0\in\Z^n \\ |x_0|<R} } 
			\left|
			\sum_{\mathfrak{l}
				\in \mathbb{Z}^n 
			} 
			\frac
			{1 \verluegoalpha{\cdot R^{N(1-2\alpha)}}}
			{(1+|\mathfrak{l}|)^{N}
			} 
			\int_{B^n_ 1\left(x_0\right)}
			\int_{\widetilde{t_0}}^{\widetilde{t_0}+1}
			\left|E'f_{\mathfrak{l}
			}(y,s)	
			\right|
			d s d y
			\right|^2.
			\intertext{
				Therefore, denoting $ C_\mathfrak{l} = \frac{1}{(1+|\mathfrak{l}|)^{N}}$ and using Cauchy-Schwarz,}
			&
			\lesssim 
			\sum_{\substack{x_0\in\Z^n \\ |x_0|<R} }  
			\sum_{\mathfrak{l}
				\in \mathbb{Z}^n 
			} 
			C_\mathfrak{l}
			\left\|
			E'f_{\mathfrak{l}
			}(y,t)
			\right\|^2_{L^2(B^{n+1}_2(x_0,\widetilde{t_0}))}	.
		\end{align}
		Let us choose $ X = \bigcup_{\substack{x_0\in\Z^n \\ |x_0|<R} } B^{n+1}_2\left(x_0,\widetilde{t_0}\right) $. The above lets us deduce
		that
		\begin{align}
			\left\|
			\sup_{0< t\leq R}|E'f(x,t)|
			\right\|_{L^2(B^n_R(0),dx)}^2
			&\lesssim 
			\sum_{\mathfrak{l}
				\in \mathbb{Z}^n 
			} 
			C_{\mathfrak{l}}
			\left\|E'f_{\mathfrak{l}
			}(y,t)
			\right\|_{L^2(X)}^2.
			\intertext{By Theorem \ref{theorem_Cor1.7 of Du Zhang modified}, this is,
			}
			&\lesssim 
			C_{\epsilon} 
			\verluegoalpha{R^{\frac{2\beta(1-2\alpha)}{n+1}}}
			\verluegoalpha{R^{\largetwoN(1-2\alpha)}}  
			\sum_{\mathfrak{l}
				\in \mathbb{Z}^n 
			} 
			C_{\mathfrak{l}}
			(\phi_{X,n})^{\frac{2}{n+1}} \|f_{\mathfrak{l}
			}
			\|_{L^2(\R^n)}^2 R^{\frac{2n}{2(n+1)}+\epsilon}.
		\end{align}
		Recall that, given $ x_0\in\Z^n\cap B_R^n(0), $ we have chosen exactly one $ \widetilde{t_0}\in \Z\cap[0,R] $. Consequently, $ \phi_{X,n} \leq 1 $. Therefore, the above inequalities yield
		\begin{equation}
			\left\|
			\sup_{0< t\leq R}|E'f(x,t)|
			\right\|_{L^2(B^n_R(0),dx)}^2
			\lesssim C_{\epsilon} 
			R^{\frac{2n}{2(n+1)}+\epsilon}\|f\|_{2}^2.
		\end{equation}
	\end{proof}
\nocite{*}
\medskip

\printbibliography
 
\footnotesize
\sc JAVIER MINGUILLÓN, DEPARTMENT OF MATHEMATICS, UNIVERSIDAD AUTÓNOMA DE MADRID, 28049 MADRID, SPAIN

\normalfont
\textit{E-mail address:} javier.minguillon@uam.es

\end{document}